\documentclass[12pt]{amsart}
\usepackage[utf8]{inputenc}
\usepackage[T1]{fontenc}
\usepackage[all]{xy}
\usepackage{lipsum}
\usepackage{url}
\usepackage{tikz}
\usepackage{stackrel}
\usepackage{color}

\usepackage{amsmath,amsthm,amssymb}
\newcommand{\aspas}[1]{``{#1}''}
\usepackage{xcolor}
\usepackage{graphicx}
\newtheorem{theorem}{Theorem}[section]
\newtheorem{lemma}[theorem]{Lemma}
\newtheorem{example}[theorem]{Example}
\newtheorem{proposition}[theorem]{Proposition}
\theoremstyle{definition}
\newtheorem{definition}[theorem]{Definition}

\newtheorem{remark}[theorem]{Remark}
\newtheorem{corollary}[theorem]{Corollary}

\numberwithin{equation}{section}

\begin{document}

%%%%%%%%%%%%%%%%%%%%%%%%%%%%%%%%%%%%%%%%%%%%%%%%%%%%%%%%%%%%

\renewcommand{\bf}{\bfseries}
\renewcommand{\sc}{\scshape}
%insert defs/styles

\title[Index and Sectional Category]%
{Index and Sectional Category}

%    Information for first author:
%    Information for first author:
\author{Cesar A. Ipanaque Zapata}
\address{Departamento de Matem\'atica, IME-Universidade de S\~ao Paulo, Rua do Mat\~ao, 1010 CEP: 05508-090, S\~ao Paulo, SP, Brazil }
\email{cesarzapata@usp.br}
\email{dlgoncal@ime.usp.br}

%    Information for second author (if needed):
\author{Daciberg L. Gonçalves} 
%\address{Departamento de Matemática - IME-USP,
%Caixa Postal 66281 - Ag. Cidade de São Paulo,
%CEP: 05314-970 - São Paulo - SP - Brasil}
%\email{dlgoncal@ime.usp.br}
%\thanks{Support information for the second author.
%\textcolor{green}{There is a text  problema here.You type \subclass[2020] and appears MSC 1991}
%    General info
%%%%%%%%%%%%%%%%%%%%%%%%%%%%%%%%%%%%%%%%%%%%%%%%%%%
%         %55r80=Discriminantal varieties, configuration spaces  
%         %55M30=Ljusternik-Schnirelman (Lyusternik-Shnirelʹman) category of a space
%         %55P10=Homotopy equivalences
%         	55M20  	Fixed points and coincidences in algebraic topology
%         %57N37=Isotopy and pseudo-isotopy
% 57M10  	Covering spaces and low-dimensional topology
%55R35  	Classifying spaces of groups and $H$-spaces in algebraic topology
%         Please use the current 2010 Mathematics Subject Classification:             %
%         http://www.ams.org/mathscinet/msc/                                                        %
%         http://www.zentralblatt-math.org/msc/en/                                                 %
%%%%%%%%%%%%%%%%%%%%%%%%%%%%%%%%%%%%%%%%%%%%%%%%%%%

\thanks {The first author would like to thank grant\#2022/16695-7, S\~{a}o Paulo Research Foundation (FAPESP) for financial support.}

\begin{abstract} Let $G$ be a finite group with order $|G|=\ell$ and $2\leq q\leq \ell$. For a free $G$-space $X$, we introduce a notion of $q$-th index of $(X,G)$, denoted by $\text{ind}_q(X,G)$. Our concept is relevant in the Borsuk-Ulam theory. We draw general estimates for the 
  $q$-th index in terms of the sectional category of the quotient map $X\to X/G$, denoted by $\text{secat}(X\to X/G)$. This property connects a standard problem in Borsuk-Ulam theory to current research trends in sectional category. Under certain hypothesis we observed that $\text{secat}(X\to X/G)=\text{ind}_2(X,G)+1$. As an application of our results, we present new results in Borsuk-Ulam theory and sectional category. 
\end{abstract}

\maketitle

%%%%%%%%%%%%%%%%%%%%%%%%%%%%%%%%%%%%%%%%%%%%%%%%%%%%%%%%%%%%%%

\section{Introduction}\label{secintro}
In this article \aspas{space} means a topological space, and by a \aspas{map} we will always mean a continuous map. Fibrations are taken in the Hurewicz sense. A $G$-space $X$ is a Hausdorff space $X$,  in which $G$  is a finite group and acts freely on $X$ on the left. The map $X\to X/G$ will be denote the quotient map. Given a free $G$-space $X$, we have that the quotient map $X\to X/G$ is a fibration \cite[Theorem 3.2.2, p. 66]{tom2008}.

\medskip Let $S^m$ be the $m$-dimensional sphere, $A:S^m\to S^m$ the antipodal involution %(i.e., $A(x)=-x$ for any $x\in S^m$) 
 and $\mathbb{R}^n$ the $n$-dimensional Euclidean space. The famous Borsuk-Ulam theorem states that for every map $f:S^m\to \mathbb{R}^m$ there exists a point $x\in S^m$ such that $f(x)=f(-x)$ \cite{borsuk1933}.
 
 \medskip The question to which the Borsuk-Ulam theorem gives a positive answer can be generalized in several ways.
 A positive response to such generalized  questions, even under certain assumptions, will lead to generalizations of the classical Borsuk-Ulam theorem.

\medskip Replacing $S^m$ together with the free involution given by the antipodal  map  by a free $G$-space $X$, and $\mathbb{R}^m$ by a Hausdorff space $Y$, given $2\leq q\leq |G|$, we say that the triple $\left(X,G;Y\right)$ satisfies \textit{the $q$-th Borsuk-Ulam property} ($q$-th BUP) if for every map $f:X\to Y$ there exists a point $x\in X$ such that there exist distinct $g_1,\ldots,g_q\in G$ such that $f(g_1x)=\cdots=f(g_qx)$. Instead of saying \aspas{the $2$-th BUP} we will say \aspas{the BUP}. Then we ask which triples $\left(X,G;Y\right)$ satisfy the $q$-th BUP, where $2\leq q\leq |G|$.

\medskip In the case $Y=\mathbb{R}^{n+1}$ if $(X,G;\mathbb{R}^{n+1})$ satisfies the $q$-th Borsuk-Ulam property, then it follows easily from the definition that 
 $(X,G;\mathbb{R}^{n})$ satisfies the $q$-th Borsuk-Ulam property. Furthermore, the triple $\left(X,G;\mathbb{R}^0\right)$ satisfies the $q$-th BUP, for any $2\leq q\leq |G|$. We define \textit{the $q$-th index of $(X,G)$} (see Definition~\ref{defn-index} together with Item (1) 
 of Remark~\ref{rem:index}) as the greatest $n\geq 0$ such that the $q$-th BUP holds for $(X,G;\mathbb{R}^n)$, where $2\leq q\leq |G|$. A major problem is to find 
 the $q$-th index of a pair $(X,G)$. 

\medskip The study of $q$-th index via sectional category for $q>2$ is still non-existent and, in fact, this work takes a first step in this direction. Nevertheless the $2$-th index of $(X,\mathbb{Z}_2)$ was recently studied in \cite{zapata2023}. In addition, in the case that $G=\mathbb{Z}_p$ with $p$ prime, the $p$-th BUP was studied in \cite[Chapter VII]{schwarz1966}. Also, the $q$-th BUP for $\left(X,\mathbb{Z}_p;\mathbb{R}^n\right)$ was studied in \cite{cohen1976}. In \cite{karasev2011}, given a manifold $X$, the authors study the $q$-th BUP for triples $\left(X,\mathbb{Z}_p;\mathbb{R}^n\right)$. 

\medskip For Hausdorff spaces $X,Y$ and a fixed-point free involution $\tau:X\to X$, in \cite{zapata2023}, the authors discover a connection between the sectional category of the double covers $X\to X/\tau$ and $F(Y,2)\to D(Y,2)$ (the projection from the ordered configuration space $F(Y,2)$ to its unordered quotient $D(Y,2)=F(Y,2)/\Sigma_2$, 
(see \cite{fadell1962configuration})), and the BUP for the triple $\left(X,\tau;Y\right)$. In \cite{zapata2023} was formulated the following conjecture.

\medskip\noindent\textbf{Conjecture.}
Let $M^m$ be a connected, $m$-dimensional CW complex and $\tau$ be a fixed-point free cellular involution defined on $M$. The $2$-th index of $(M,\tau)$ is equal to $\mathrm{secat}(X\to X/\tau)-1$, equivalently, by \cite[Theorem 3.30]{zapata2023}, the triple $\left((M,\tau);\mathbb{R}^{\mathrm{secat}(q)}\right)$ does not satisfy the BUP.

\medskip The main results of this note are as follows. In Proposition~\ref{top-bup} we present a topological criterion for the $q$-th BUP. Proposition~\ref{firs-estimation} provides a first estimation of the index $\text{ind}_{q}(X,G)$. We establish a connection between sectional category and the $q$-th BUP (Theorem~\ref{bup-secat}). For instance, we obtain a characterization of the $q$-th BUP in terms of sectional category (Theorem~\ref{thm:q-bup-scat}). Theorem~\ref{theorem:general-secat-index} presents an upper bound for sectional category in terms of the index. Theorem~\ref{thm:lower-bound-index} presents a lower bound for sectional category in terms of the index. A relationship between $\mathrm{secat}\left(X\to X/\mathbb{Z}_p\right)$ and $\text{ind}_q\left(X,\mathbb{Z}_p\right)$ is presented in Theorem~\ref{prop:secat-indexp}. As applications of our results, we provide a positive answer to the conjecture stated above (Corollary~\ref{cor:conj}), under  weaker hypothesis 
on $M^m$. Corollary~\ref{cor:lower-secat-lower-also-index} shows that lower secat implies lower index. In addition, we use Borsuk-Ulam theory (index) to present new results in sectional category theory. For instance, we describe within one the value of secat of a double covering $\pi:S_1 \to S_2$ between any two closed surfaces (Corollary~\ref{cor}). Furthermore, we compute secat of double covering of $3$-manifolds having geometry  $S^2\times \mathbb{R}$ (Corollary~\ref{cor:s3-r-geometry}), and of   covering $\pi:S_1 \to S_2$ between any two closed surfaces obtained by a free action of the cyclic group $\mathbb{Z}_n$ (Corollary \ref{corz_n}). 

\medskip The paper is organized as follows: In Section \ref{bu}, we recall a natural generalization of the Borsuk-Ulam property. We also introduce the notion of $q$-th index (Definition~\ref{defn-index}). Proposition~\ref{top-bup} presents a topological criterion for the $q$-th BUP. In Section \ref{sn}, we begin by recalling the notion of sectional category and basic results about this numerical invariant. Theorem~\ref{thm:q-bup-scat} presents a characterization of the $q$-th BUP in terms of sectional category. In particular, Proposition~\ref{prop:p-bup-secat} implies a relationship between $\mathrm{secat}\left(X\to X/\mathbb{Z}_p\right)$ and $\text{ind}_p\left(X,\mathbb{Z}_p\right)$. As applications of our results, in Section~\ref{sec:appl}, we provide  a positive answer to the conjecture stated above, in fact under weaker hypothesis than the ones  stated  there. Then we compute secat of double covering of closed surfaces and double covering of the  $3$-manifolds having geometry $S^2\times \mathbb{R}$.

%\medskip The authors of this paper deeply thank the referees for very valuable comments and timely corrections on previous versions of the work.
%%%%%%%%%%%%%%%%%%%%%%%%%%%%%%%%%%%%%%%%%%%%%%%%%%%%%%%%%%%%%%%%%%%%%%%%%%%%%%%%%%%%%%%%%%%%%%%%555
\section{$q$-th Borsuk-Ulam property and $q$-th index}\label{bu}
 Let $X$ be  a free  $G$-space for  $G$ a finite group. Recall that, the triple $\left(X,G;Y\right)$ satisfies \textit{the $q$-th Borsuk-Ulam property} ($q$-th BUP) if for every map $f:X\to Y$ there exists a point $x\in X$ such that there exist distinct $u_1,\ldots,u_q\in G$ such that $f(u_1x)=\cdots=f(u_qx)$. Instead of saying \aspas{the $2$-th BUP} we will say \aspas{the BUP}.

\medskip We have the following special cases. 
\begin{remark}
    Let $X$ be a free $G$-space and $Y$ be a Hausdorff space.
    \begin{enumerate}
 \item[(1)] $q=|G|$, the triple $\left(X,G;Y\right)$ satisfies the $q$-th BUP if and only if for every map $f:X\to Y$ there exists a point $x\in X$ such that 
  $f(x)=f(gx)$ for all $g\in G$.
         \item[(2)] $q=2$, the triple $\left(X,G;Y\right)$ satisfies the $2$-th BUP if and only if for every map $f:X\to Y$ there exists a point $x\in X$ such that there exist distinct $u,v\in G$ with $f(ux)=f(vx)$. Equivalently, for every map $f:X\to Y$ there exists a point $x\in X$ such that the restriction $f_{| O(x)}:O(x)\to Y$ is not injective, where $O(x)=\{gx:~g\in G\}$ denotes the orbit of $x$.
    \end{enumerate}
\end{remark}

%\begin{example}\label{exam:0-case}
% Let $X$ be  a free  $G$-space for  $G$ a finite group. The triple $\left(X,G;\mathbb{R}^0\right)$ satisfies the $q$-th BUP, for any   
%\end{example}

Let $G$ be a finite group with order $|G|=\ell$, $2\leq q\leq \ell$ and $Y$ be a space. The \textit{configuration-like space} $F(Y,\ell,q)$ defined by Cohen  and Lusk in \cite[Pg. 314]{cohen1976} (see also \cite[Definition 4.3, p. 1040]{karasev2011}) is given by  \begin{align*}
     F(Y,\ell,q)&=\{(y_1,\ldots,y_\ell)\in Y^\ell: \text{ for any $\{y_{i_1},\ldots,y_{i_q}\}$ with $0<i_1<\cdots<i_q\leq \ell$,}\\
     &\quad\text{ at least 2 of the $y_{i_j}$'s are different}\}.\\
 \end{align*} Set $G=\{g_1,\ldots,g_\ell\}$ and $h:G\to [\ell]$ a bijection given by $h(g_j)=j$, here $[\ell]=\{1,\ldots,\ell\}$. Set $\mathrm{Map}(G,Y)$ be the set of maps $G\to Y$. We have a bijection $\varphi:Y^\ell\to \mathrm{Map}(G,Y),~\varphi(y_1,\ldots,y_\ell)(g)=y_{h(g)}$, whose inverse is $\psi:\mathrm{Map}(G,Y)\to Y^\ell,~\psi(\phi)=(\phi(g_1),\ldots,\phi(g_\ell))$. Note that $\varphi(F(Y,\ell,q))=\mathrm{Map}_q(G,Y)$, where $\mathrm{Map}_q(G,Y)=\{\phi\in \mathrm{Map}(G,Y):~\phi_{| S} \text{ is not constant, $\forall S\subset G$ with $|S|=q$}\}$.
 
 \medskip We will consider the left action of $G$ on $\mathrm{Map}(G,Y)$ given by $(g\phi)(h)=\phi(g^{-1}h), \forall g,h\in G, \phi\in \mathrm{Map}(G,Y)$ (\cite[Definition 4.5, p. 1041]{karasev2011}). This action restricts to a free left action of $G$ on $\mathrm{Map}_q(G,Y)$ whenever $q\leq e_\ell$, where $e_\ell=\min\{m:~\text{ $m$ is a prime divisor of $\ell=|G|$}\}$. For instance, $e_{p^n}=p$ with $p$ prime and $n\geq 1$. In addition, given a map $f:X\to Y$, we have a $G$-equivariant map  $\hat{f}:X\to\mathrm{Map}(G,Y)$ given by $\hat{f}(x)(g)=f(g^{-1}x)$ (see \cite[Definition 6.3, p. 1049]{karasev2011}).

\medskip Define  a left action of $G$ on $Y^\ell$  by \begin{equation}\label{eq:action}
g(y_1,y_2,\ldots,y_\ell)=\left(y_{h(g^{-1}g_1)},\ldots,y_{h(g^{-1}g_\ell)}\right). \end{equation} This action restricts to a free left action of $G$ on $F(Y,\ell,q)$ whenever $q\leq e_\ell$. Note that $\varphi$ and $\psi$ are $G$-equivariant maps. In addition, given a map $f:X\to Y$, we have a $G$-equivariant map  $\overline{f}:X\to Y^\ell$ given by \begin{equation}\label{eq:equiv}
  \overline{f}(x)=\left(f(g_1^{-1}x),\ldots,f(g_{\ell}^{-1}x))\right)  
\end{equation}  

\medskip We will check the following topological criterion for the $q$-th BUP. 

\begin{proposition}\label{top-bup} 
The triple $\left(X,G;Y\right)$ does not satisfy the $q$-th BUP if and only if there exists a $G$-equivariant map $X\to F(Y,\ell,q)$. 
\end{proposition}
\begin{proof}
Suppose first that $\left(X,G;Y\right)$ does not satisfy the $q$-th BUP. Then there exists a map $f:X\to Y$ such that for all distinct $u_1,\ldots,u_q\in G$ we have that at least 2 of the $f(u_jx)$'s are different for all $x\in X$. Then $\overline{f}(x)=\left(f(g_1^{-1}x),\ldots,f(g_{\ell}^{-1}x))\right)\in F(Y,\ell,q)$, for any $x\in X$. Take the map $\overline{f}:X\to F(Y,\ell,q)$ and recall that $\overline{f}$ is $G$-equivariant (see Equation~(\ref{eq:equiv})).  

We now prove the converse. Suppose that there exists a such $G$-equivariant map $\Phi:X\to F(Y,\ell,q)$. Consider the $G$-equivariant map $\widetilde{\Phi}:X\stackrel{\Phi}{\to}F(Y,\ell,q)\stackrel{\varphi}{\to}\mathrm{Map}_q(G,Y)$. Set $f:X\to Y$ given by $f(x)=\widetilde{\Phi}(x)(g_1)$. Note that $f(x)=\Phi_1(x)$, where $\Phi=(\Phi_1,\ldots,\Phi_\ell)$. Observe that, for any distinct $u_1,\ldots,u_q\in G$ and $x\in X$, we have:\begin{align*}
    f(u_jx)&=\widetilde{\Phi}(u_jx)(g_1) &\\
    &=\left(u_j\widetilde{\Phi}(x)\right)(g_1) &\mbox{($\widetilde{\Phi}$ is $G$-equivariant)}\\
    &=\widetilde{\Phi}(x)(u_j^{-1}g_1),
\end{align*} for each $j=1,\ldots,q$, and hence at least 2 of the $f(u_jx)$'s are different. Therefore, the triple $\left(X,G;Y\right)$ does not satisfy the $q$-th BUP.
\end{proof}

\medskip In the case $Y=\mathbb{R}^n$ a major problem is to find the greatest $n$ such that the $q$-th BUP holds for a specific $(X,G)$, $2\leq q\leq |G|$. 

\begin{definition}\label{defn-index} Let $G$ be a finite group with order $|G|=\ell$, $2\leq q\leq \ell$ and $X$ be a free $G$-space. The \textit{$q$-th index} of $(X,G)$, denoted by $\text{ind}_{q}(X,G)$, is the least integer $k\in \{0,1, 2, \ldots\}$ such that there exists a $G$-equivariant map $X\to F(\mathbb{R}^{k+1},\ell, q)$. We set $\text{ind}_{q}(X,G)=\infty$ if no such $k$ exists.  
\end{definition}

\begin{remark}\label{rem:index}
 \noindent
 \begin{enumerate}
     \item[(1)] Note that the greatest integer $n\geq 0$ such that $\left(X,G;\mathbb{R}^n\right)$ satisfies the $q$-th BUP coincides with $\text{ind}_{q}(X,G)$ (it follows from Proposition~\ref{top-bup}). 
     \item[(2)] Furthermore, $\text{ind}_{q}(X,G)\geq \text{ind}_{q+1}(X,G)$. 
    \item[(3)] Our $\text{ind}_{2}(X,\mathbb{Z}_2)$ coincides with the $\mathbb{Z}_2$-index from \cite[Definition 2.2, p. 41]{zapata2023}.  
\item[(4)] The connectivity of  $F(\mathbb{R}^{k+1},\ell, q)$ is known to be $(k+1)(q-1)-2$ (it follows from \cite[p. 72]{guillemin1974differential}). For $q=2$ this also follows from \cite[Corollary 2.1, p. 113]{fadell1962configuration}.  %eq on page 13 line 11 of \cite{fahus}. 
 \end{enumerate}
\end{remark}

The following proposition provides a first estimation of this new index $\text{ind}_{q}(X,G)$, which generalizes \cite[Lemma 2.4  (3)]{bghz}. %We estimate  the following cases: a) $G$ an arbitrary finite group and $q=2$; b) The case              
%where $2\leq q  \leq e$ where  $e$      is  the smallest prime which divides the order of the group $G$.    In the particular case that $G$ is a cyclic group of order
%$p$ for  $p$  a prime number, then we have $2\leq q\leq p$.
\begin{proposition}\label{firs-estimation} Let $X$ be a $CW$-complex of dimension $n$ which admits a cellular free left action of $G$. Let $2\leq q  \leq e_\ell$ where $e_\ell=\min\{m:~\text{ $m$ is a prime divisor of $\ell=|G|$}\}$. %$e_\ell$ is  the smallest prime which divides the order of the group $G$. Then:\\
%\begin{enumerate}
%\item $\text{ind}_{2}(X,G)\leq n$.
%\item   Then $\text{ind}_{q}(X,G)\leq (n+1)/(q-1)-1$.\\
Then the following inequality holds \[\mathrm{ind}_{q}(X,G)\leq  \dfrac{n+1}{q-1}-1.\]
%\end{enumerate} 
\end{proposition}  
\begin{proof}
    First observe that the hypothesis  $2\leq q \leq e_\ell$  implies that $G$ acts freely on $ F(\mathbb{R}^{k+1},\ell, q)$ (in  fact  if and  only  if). Then we have a finite covering  $ F(\mathbb{R}^{k+1},\ell, q)\to F(\mathbb{R}^{k+1},\ell, q)/G$ and a homotopy fibration $ F(\mathbb{R}^{k+1},\ell, q)\to F(\mathbb{R}^{k+1},\ell, q)/G \to  K(G,1)$, where the connectivity of the homotopy fibre, which has the same homotopy type as $ F(\mathbb{R}^{k+1},\ell, q)$, is $(k+1)(q-1)-2.$   So, by the relation between the dimension of the space  $X/G$ and the connectivity of the fibre, using standard argument of obstruction  theory,  the map,  $X/G\to  K(G,1)$ provided  by the covering  map $X\to X/G$  lifts through the map  $F(\mathbb{R}^{k+1},\ell, q)/G \to  K(G,1)$, and we obtain a map $\psi: X/G\to F(\mathbb{R}^{k+1},\ell,q)/G,$ which makes the diagram below homotopy  commutative.
\begin{equation*}\label{diag_equiv2}\begin{gathered}\xymatrix{
X   \ar[d]_-{ }& &   F(\mathbb{R}^{k+1},\ell, q) \ar[d]^-{ } \\
 X/G \ar@{.>}[rr]^-{\psi} \ar[rd]_-{} &  &   F(\mathbb{R}^{k+1},\ell, q) /G
%\pi_1 ( N_{\tau_1}, p_{\tau_1} (n_1) )
\ar[ld]^-{}  \\
& K(G,1)&
}\end{gathered}\end{equation*}
	The  constructed  map $\psi$ which makes the lower triangle homotopy commutative,  by basic covering space theory,  admits a lift $\hat \psi:X\to F(\mathbb{R}^{k+1},\ell, q) $ which is a $G$-map  and the result follows.   
 \end{proof}

For any free $G$-space $X$, the quotient map $X\to X/G$ is a $|G|$-sheeted covering map. Recall that a covering map is a locally trivial bundle whose fiber is a discrete space \cite[Example 4.5.3, p. 126]{aguilar2002} and thus the quotient map $X\to X/G$ is a principal fibration in the sense of Schwarz \cite[p. 59]{schwarz1966}.  

\medskip From \cite[p. 61]{schwarz1966} we have the following remark. 

\begin{remark}\label{bup-pullback} Let $X, Y$ be free $G$-spaces. Note that, any commutative diagram in the form  \begin{eqnarray*}
\xymatrix{ 
       X  \ar[r]^{\varphi}\ar[d]_{} &  Y\ar[d]^{} &\\
      X/G \ar[r]_{\overline{\varphi} } & Y/G &
       }
\end{eqnarray*} 
%\end{gathered}\end{equation}

\noindent where $\varphi:X\to Y$ is a $G$-equivariant map and $\overline{\varphi}$ is induced by $\varphi$ in the quotient spaces, is a pullback since $\varphi$ restricts to a homeomorphism on each fiber (in this case both $X\to X/G$ and $Y\to Y/G$ are $|G|$-sheeted covering maps).
\end{remark}

%%%%%%%%%%%%%%%%%%%%%%%%%%%%%%%%%%%%%%%%%%%%%%%%%%%%%%%%%%%%%%%%%%%%%%%%%%%%%%%%%%%%%%%%%%%%%%%%%%%%%%%%%%%%%%%5

\section{Sectional category and the Borsuk-Ulam property}\label{sn}
In this section we begin by recalling the notion of sectional category together with basic results about this numerical invariant. Then we derive results related with the Borsuk-Ulam property. We shall follow the terminology in \cite{zapata2023}. If $f$ is homotopic to $g$ we shall denote by $f\simeq g$. The map $1_Z:Z\to Z$ denotes the identity map. 

\medskip Let $p:E\to B$ be a fibration.  A \textit{cross-section} or \textit{section} of $p$ is a right inverse of $p$, i.e., a map $s:B\to E$, such that $p\circ s = 1_B$ . Moreover, given a subspace $A\subset B$, a \textit{local section} of $p$ over $A$ is a section of the restriction map $p_|:p^{-1}(A)\to A$, i.e., a map $s:A\to E$, such that $p\circ s$ is the inclusion $A\hookrightarrow B$.

\medskip We recall the following definition, see \cite{schwarz1958genus} or \cite{schwarz1966}.
\begin{definition}
   The \textit{sectional category} of $p$, called originally by Schwarz genus of $p$, see \cite[Definition 5, section 1 chapter III]{schwarz1966}, and denoted by $\mathrm{secat}\hspace{.1mm}(p),$ is the minimal cardinality of open covers of $B$, such that each element of the cover admits a local section to $p$. We set $\mathrm{secat}\hspace{.1mm}(p)=\infty$ if no such finite cover exists.
\end{definition}
%
%\textcolor{blue}{For a commutative ring $R$ and a proper ideal $S\subset R$, the \textit{nilpotency index} of $S$ is given by \[\text{nil}(S)=\min\{k:~\text{ any product of $k$ elements of $S$ is trivial}\}.\] Note that, $\text{nil}(S)$ coincides with $n+1$, where $n$ is the maximum number of factors in a nonzero product of elements from $S$.}
%
%\textcolor{blue}{The following statement gives a lower bound of sectional category in terms of any multiplicative cohomology (see \cite[Proposiç\~{a}o 4.3.17-(3), pg. 138]{zapata2022}).}
%
%\textcolor{blue}{\begin{lemma}\label{prop-sectional-category}
%Let $h^\ast$ be a multiplicative cohomology theory and $p:E\to B$ be a fibration, then \[\mathrm{secat}\hspace{.1mm}(p)\geq \mathrm{nil}(\mathrm{Ker}(p^\ast)),\] where $p^\ast:h^\ast(B)\to h^\ast(E)$ is the induced homomorphism in cohomology.
%\end{lemma}}
%
%\textcolor{blue}{\begin{remark}
%Lemma~\ref{prop-sectional-category} implies that if there exist cohomology classes $\alpha_1,\ldots,\alpha_k\in h^\ast(B)$ with $p^\ast(\alpha_1)=\cdots=p^\ast(\alpha_k)=0$ and $\alpha_1\cup\cdots\cup \alpha_k\neq 0$, then $\mathrm{secat}\hspace{.1mm}(p)\geq k+1$. In this paper we will use Lemma~\ref{prop-sectional-category} for $h^\ast$ as being the singular cohomology with any coefficient ring (as was presented by James in \cite[pg. 342]{james1978}).
%\end{remark} }

\medskip Now, note that, if the following diagram

\begin{eqnarray*}
\xymatrix{ E^\prime \ar[rr]^{\,\,} \ar[dr]_{p^\prime} & & E \ar[dl]^{p}  \\
        &  B & }
\end{eqnarray*}
commutes up homotopy, then $\mathrm{secat}\hspace{.1mm}(p^\prime)\geq \mathrm{secat}\hspace{.1mm}(p)$ (see \cite[Proposition 6, p. 70]{schwarz1966}). Also, from \cite[p. 1619]{zapata2022higher}, a \textit{quasi pullback} means a strictly commutative diagram
\begin{eqnarray*}%\label{xfy}
\xymatrix{ \rule{3mm}{0mm}& X^\prime \ar[r]^{\varphi'} \ar[d]_{f^\prime} & X \ar[d]^{f} & \\ &
       Y^\prime  \ar[r]_{\,\,\varphi} &  Y &}
\end{eqnarray*} 
such that, for any strictly commutative diagram as the one on the left hand-side of~(\ref{diagramadoble}), there exists a (not necessarily unique) map $h:Z\to X^\prime$ that renders a strictly commutative diagram as the one on the right hand-side of~(\ref{diagramadoble}). 
\begin{eqnarray}\label{diagramadoble}
\xymatrix{
Z \ar@/_10pt/[dr]_{\alpha} \ar@/^30pt/[rr]^{\beta} & & X \ar[d]^{f}  & & &
Z\rule{-1mm}{0mm} \ar@/_10pt/[dr]_{\alpha} \ar@/^30pt/[rr]^{\beta}\ar[r]^{h} & 
X^\prime \ar[r]^{\varphi'} \ar[d]_{f^\prime} & X \\
& Y^\prime  \ar[r]_{\,\,\varphi} &  Y & & & & Y^\prime &  \rule{3mm}{0mm}}
\end{eqnarray}   

Note that such a condition amounts to saying that $X'$ contains the canonical pullback $Y'\times_Y X$ determined by $f$ and $\varphi$ as a retract in a way that is compatible with the mappings into $X$ and $Y'$.

\medskip Next, we recall the notion of LS category which, in our setting, is one bigger than the one given in \cite[Definition 1.1, p.1]{cornea2003lusternik}. 

\begin{definition}
The \textit{Lusternik-Schnirelmann category} (LS category) or category of a topological space $X$, denoted by cat$(X)$, is the least integer $m$ such that $X$ can be covered by $m$ open sets, all of which are contractible within $X$. We set $\text{cat}(X)=\infty$ if no such $m$ exists.
\end{definition}

We have $\text{cat}(X)=1$ if and only if $X$ is contractible. The LS category is a homotopy invariant, i.e., if $X$ is homotopy equivalent to $Y$ (which we shall denote by $X\simeq Y$), then $\text{cat}(X)=\text{cat}(Y)$.
%
%\textcolor{blue}{Furthermore, the invariant $\mathrm{cat}(-)$ satisfies the following properties.}
%
%\textcolor{blue}{\begin{lemma}\label{cat-stimates}
%\noindent
%\begin{enumerate}
%    \item \cite[Proposition 5.1, pg. 336]{james1978} If $X$ is a $(q-1)$-connected CW complex ($q\geq 1$), then \[\mathrm{cat}(X)\leq \dfrac{\mathrm{hdim}(X)}{q}+1,\]  where $\mathrm{hdim}(X)$ denotes the homotopical dimension of $X$, i.e., the minimal dimension of CW complexes having the homotopy type of $X$. 
%    \item \cite[Proposiç\~{a}o 4.1.34, pg. 108]{zapata2022} We have $$\mathrm{cat}(X)\geq \mathrm{nil}\left(\widetilde{h}^\ast(X)\right),$$ where $\widetilde{h}^\ast(X)$ is any multiplicative reduced cohomology theory.
%\end{enumerate}
%\end{lemma}}

\medskip For convenience, we record the following standard properties: 

 \begin{lemma}\label{prop-secat-map}
 Let $p:E\to B$ be a fibration.
 \begin{enumerate}
 \item \cite[Lemma 3.5]{zapata2023} If the following square 
\begin{eqnarray*}
\xymatrix{ E^\prime \ar[r]^{\,\,} \ar[d]_{p^\prime} & E \ar[d]^{p} & \\
       B^\prime  \ar[r]_{\,\, } &  B &}
\end{eqnarray*}
is a quasi pullback, then $\mathrm{secat}\hspace{.1mm}(p^\prime)\leq \mathrm{secat}\hspace{.1mm}(p)$. 
\item \cite[Proposition 3.6]{zapata2023} For any space $Z$, we have \[\mathrm{secat}(p\times 1_Z)=\mathrm{secat}(p).\]  
\item \cite[Proposition 5.1, p. 336]{james1978} If $X$ is a $(q-1)$-connected CW complex ($q\geq 1$), then \[\mathrm{cat}(X)\leq \dfrac{\mathrm{hdim}(X)}{q}+1,\]  where $\mathrm{hdim}(X)$ denotes the homotopical dimension of $X$, i.e., the minimal dimension of CW complexes having the homotopy type of $X$. 
\item \cite[Theorem 18, p. 108]{schwarz1966} We have $\mathrm{secat}\hspace{.1mm}(p)\leq \mathrm{cat}(B)$. 
 \end{enumerate}
\end{lemma}

\begin{remark}\label{trivial-crosssection} Given a free $G$-space $X$ for $G$ finite, we have that:
\begin{itemize}
    \item[(i)] The quotient map $ X\to X/G$ is a fibration \cite[Theorem 3.2.2, p. 66]{tom2008}. Therefore, it is possible to speak of its sectional category.
    \item[(ii)] The quotient map $X\to X/G$ is a fiber bundle \cite{aguilar2002}, and it is the trivial bundle if and only if it admits a global cross-section \cite[p. 36]{steenrod1951}. In particular, in the case that $X$ is path-connected, the quotient map $X\to X/G$ is not the trivial bundle and does not admit a global cross-section. 
\end{itemize}
\end{remark}

\begin{example}\label{exam:s1-zp}
  Let $p$ be a prime number. Consider a free action of $\mathbb{Z}_p$ on $S^1$. We have that $\mathrm{secat}(S^1\to S^1/\mathbb{Z}_p)=\mathrm{cat}\left(S^1/\mathbb{Z}_p\right)=2$. Indeed, by Item (ii) from Remark~\ref{trivial-crosssection}, the inequality $\mathrm{secat}(S^1\to S^1/\mathbb{Z}_p)\geq 2$ holds. On the other hand, we have:
  \begin{align*}
      \mathrm{secat}(S^1\to S^1/\mathbb{Z}_p)&\leq \mathrm{cat}\left(S^1/\mathbb{Z}_p\right)& \hbox{ (Item (iv) from Lemma~\ref{prop-secat-map})}\\
      &\leq 2 & \hbox{ (Item (iii) from Lemma~\ref{prop-secat-map})}.
  \end{align*} Hence, $\mathrm{secat}(S^1\to S^1/\mathbb{Z}_p)=\mathrm{cat}\left(S^1/\mathbb{Z}_p\right)=2$. 
\end{example}

The paracompactness hypothesis in \cite[Theorem 3.14]{zapata2023} was used to have that the quotient map $ X\to X/G$ is a fibration and hence it is possible to speak of its sectional category and we can use Item (1) from Lemma~\ref{prop-secat-map}. From Item (i) of Remark~\ref{trivial-crosssection}, the paracompactness hypothesis in \cite[Theorem 3.14]{zapata2023} is not necessary. So we can state the following statement without the hypotheses of paracompactness which also generalizes \cite[Theorem 3.14]{zapata2023}, even under the $q$-th BUP setting. Recall that $e_\ell=\min\{m:~\text{$m$ is a prime divisor of $\ell=|G|$}\}$.

\begin{theorem}\label{bup-secat}
Let $G$ be a finite group with order $|G|=\ell$, $2\leq q\leq e_\ell$. Suppose that $X$ is a free $G$-space and $Y$ is a Hausdorff space. If the triple $\left(X,G;Y\right)$ does not satisfy the $q$-th BUP then \[\mathrm{secat}\left(X\to X/G\right)\leq \mathrm{secat}\left(F(Y,\ell,q)\to F(Y,\ell,q)/G\right).\] Equivalently, if $\mathrm{secat}\left(X\to X/G\right)> \mathrm{secat}\left(F(Y,\ell,q)\to F(Y,\ell,q)/G\right)$ then the triple $\left(X,G;Y\right)$ satisfies the $q$-th BUP.
\end{theorem}
\begin{proof}
    It follows from Proposition~\ref{top-bup} together with Remark~\ref{bup-pullback} and Item (1) of Lemma~\ref{prop-secat-map}. Here, we use that $2\leq q\leq e_\ell$ and thus $G$ acts freely on $F(Y,\ell,q)$; and therefore the quotient map $F(Y,\ell,q)\to F(Y,\ell,q)/G$ is a fibration (see Item (i) from Remark~\ref{trivial-crosssection}). 
\end{proof}

We recall the following statement.

\begin{lemma}\label{lem:lem}\cite[Theorem 15, pg. 97]{schwarz1966} 
  Let $\theta(S,R,G,\pi)$ be a principal fibration with the total space $S$ of connectivity at least $n-2$ and $\beta(E,B,G,p)$ be a principal fibration where $B$ is a paracompact space. If $\text{secat}(p)\leq n$ then there is a $G$-equivariant map $\varphi:E\to S$.
\end{lemma}

%From Lemma~\ref{lem:lem} together with Theorem~\ref{bup-secat}, we have the following statement which presents a characterization of the $q$-th BUP in terms of sectional category. 

%\begin{proposition}\label{prop:conect-bup}
%Let $G$ be a finite group with order $|G|=\ell$, $2\leq q\leq \red{e_\ell}$. Suppose that $X$ is a paracompact free $G$-space and $Y$ is a Hausdorff space such that the connectivity of $F(Y,\ell,q)$ is at least $\mathrm{secat}\left(F(Y,\ell,q)\to F(Y,\ell,q)/G\right)-2$. We have that $\mathrm{secat}\left(X\to X/G\right)\leq \mathrm{secat}\left(F(Y,\ell,q)\to F(Y,\ell,q)/G\right)$ if and only if the triple $\left((X,G);Y\right)$ does not satisfy the $q$-th BUP. Equivalently, the triple $\left((X,G);Y\right)$ satisfies the $q$-th BUP if and only if $\mathrm{secat}\left(X\to X/G\right)> \mathrm{secat}\left(F(Y,\ell,q)\to F(Y,\ell,q)/G\right)$.
%\end{proposition}

In particular, under certain conditions, we obtain a characterization of the $q$-th BUP in terms of sectional category. 

\begin{theorem}\label{thm:q-bup-scat}
    Let $G$ be a finite group with order $|G|=\ell$, and $2\leq q\leq e_\ell$. Suppose that $X$ is a paracompact free $G$-space and $Y$ is a Hausdorff space such that the connectivity of $F(Y,\ell,q)$ is at least $\mathrm{secat}\left(F(Y,\ell,q)\to F(Y,\ell,q)/G\right)-2$. We have that the following statements are equivalent.
   \begin{enumerate}
       \item[(i)] $\mathrm{secat}\left(X\to X/G\right)> \mathrm{secat}\left(F(Y,\ell,q)\to F(Y,\ell,q)/G\right)$.
        \item[(ii)] the triple $\left(X,G;Y\right)$ satisfies the $q$-th BUP.
   \end{enumerate}
\end{theorem}
\begin{proof}
  (i) implies (ii) follows from Theorem~\ref{bup-secat}. The converse follows from Item (iii) of Lemma~\ref{lem:lem} and Proposition \ref{top-bup}. 
\end{proof}

As noted by Karasev-Volovikov in \cite[p. 1039]{karasev2011} the sectional category coincidences with the fixed point free genus in the sense of \cite[p. 1039]{karasev2011}. Hence, from \cite[Theorem 5.2, p. 1043]{karasev2011} we have the following result.

\begin{proposition}\label{compu-secat-conf}
    Let $2\leq q\leq p$, with $p$ prime. Then \[\mathrm{secat}\left(F(\mathbb{R}^n,p,q)\to F(\mathbb{R}^n,p,q)/\mathbb{Z}_p\right)\leq(n-1)(p-1)+q-1, \text{ for any $n\geq 1$}.\]
\end{proposition}

Unfortunately in general the sectional category of the fibration $F(\mathbb{R}^n,p,q)\to F(\mathbb{R}^n,p,q)/\mathbb{Z}_p$ has not yet been determined for all $n$ and $2\leq q\leq p$, with $p$ prime. However, from \cite{karasev2011} we have the following remark.
\begin{remark}\label{rem:compu-secat-conf-like}
  Let $2\leq q\leq p$, with $p$ prime. If $q>p/2$ or $q=2$ then \[\mathrm{secat}\left(F(\mathbb{R}^n,p,q)\to F(\mathbb{R}^n,p,q)/\mathbb{Z}_p\right)=(n-1)(p-1)+q-1.\] Indeed, the case $q>p/2$ follows from \cite[Theorem 5.5, p. 1045]{karasev2011}. The case $q=2$ follows from Proposition~\ref{compu-secat-conf} together with \cite[Lemma 5.9, p. 1047]{karasev2011}.  
\end{remark}

Theorem~\ref{bup-secat} together with Proposition~\ref{compu-secat-conf} imply the following result which presents an upper bound for sectional category in terms of the index.

\begin{theorem}\label{theorem:general-secat-index}
  Let $2\leq q\leq p$, with $p$ prime. Suppose that $X$ is a free $\mathbb{Z}_p$-space and $n_q=\mathrm{ind}_q(X,\mathbb{Z}_p)$. Then \[\mathrm{secat}(X\to X/\mathbb{Z}_p)\leq n_q(p-1)+q-1.\] %$q>p/2$ or $q=2$. Set
\end{theorem}
\begin{proof}
  If $n_q=\infty$ the inequality is obvious. So, we suppose that $n_q<\infty$. By Definition~\ref{defn-index}, we obtain that there exists a $\mathbb{Z}_p$-equivariant map $X\to F(\mathbb{R}^{n_q+1},p,q)$. Then, \begin{align*}
      \mathrm{secat}(X\to X/\mathbb{Z}_p)&\leq \mathrm{secat}\left(F(\mathbb{R}^{n_q+1},p,q)\right)& \hbox{(by Theorem~\ref{bup-secat})}\\
      &\leq n_q(p-1)+q-1&\hbox{(by Proposition~\ref{compu-secat-conf})}.
  \end{align*}
\end{proof}

A direct consequence of Theorem~\ref{theorem:general-secat-index} is the following result.

\begin{proposition}\label{secat-q-index-1}
  Let $2\leq q\leq p$, with $p$ prime. Suppose that $X$ is a free $\mathbb{Z}_p$-space. If $\mathrm{secat}(X\to X/\mathbb{Z}_p)\geq q$ then  $\mathrm{ind}_q(X,\mathbb{Z}_p)\geq 1$. 
\end{proposition}

Let $G$ be a finite group with order $|G|=\ell$, $2\leq q\leq \ell$. Note that, the space $F(\mathbb{R}^n,\ell,q)$ is a complement to a system of $(\ell-q+1)n$-dimensional linear subspaces in $\mathbb{R}^{\ell n}$. Thus, by \cite[p. 72]{guillemin1974differential}, it is $((q-1)n-2)$-connected. Hence, Lemma~\ref{lem:lem} implies the following statement.

\begin{proposition}\label{prop-lower-bound}
Let $G$ be a finite group with order $|G|=\ell$, $2\leq q\leq e_\ell$. Suppose that $X$ is a paracompact free $G$-space. If $\mathrm{secat}\left(X\to X/G\right)\leq n(q-1)$ then the triple $\left(X,G;\mathbb{R}^n\right)$ does not satisfy the $q$-th BUP. Equivalently, if the triple $\left(X,G;\mathbb{R}^n\right)$ satisfies the $q$-th BUP then $\mathrm{secat}\left(X\to X/G\right)\geq n(q-1)+1$.% In particular, $\text{ind}_2(X,\mathbb{Z}_p)\leq n-1$.
\end{proposition}

Proposition~\ref{prop-lower-bound} implies the following result which presents a lower bound for sectional category in terms of the index.

\begin{theorem}\label{thm:lower-bound-index}
    Let $G$ be a finite group with order $|G|=\ell$, $2\leq q\leq e_\ell$. Suppose that $X$ is a paracompact free $G$-space. Set $n_q=\mathrm{ind}_q\left(X,G\right)<\infty$. Then 
    \[n_q(q-1)+1\leq\mathrm{secat}\left(X\to X/G\right).\]
\end{theorem}
\begin{proof}
    By definition of index, the triple $\left(X,G;\mathbb{R}^{n_q}\right)$ satisfies the $q$-th BUP. Then, by Proposition~\ref{prop-lower-bound}, we obtain that $\mathrm{secat}\left(X\to X/G\right)\geq n_q(q-1)+1$.
\end{proof}

A direct application of Theorem~\ref{thm:lower-bound-index} says that lower sectional category implies lower index. 

\begin{proposition}\label{prop:lower-secat-index}
       Let $G$ be a finite group with order $|G|=\ell$, $2\leq q\leq e_\ell$. Suppose that $X$ is a paracompact free $G$-space and $\mathrm{ind}_q\left(X,G\right)<\infty$.
       \begin{enumerate}
           \item[(1)] If  $\mathrm{secat}\left(X\to X/G\right)=2$ then $\mathrm{ind}_q\left(X,G\right)\in\{0,1\}$. In addition, $\mathrm{ind}_q\left(X,G\right)=0$ whenever $q\geq 3$.
           \item[(2)] If  $\mathrm{secat}\left(X\to X/G\right)=3$ then $\mathrm{ind}_q\left(X,G\right)\in\{0,1,2\}$. In addition, $\mathrm{ind}_3\left(X,G\right)\in\{0,1\}$ and $\mathrm{ind}_q\left(X,G\right)=0$ whenever $q\geq 4$.
       \end{enumerate} 
\end{proposition}

For instance, recall that the connectivity of $F(\mathbb{R}^n,p,p)$ is equal to $n(p-1)-2$ and $\mathrm{secat}\left(F(\mathbb{R}^n,p,p)\to F(\mathbb{R}^n,p,p)/\mathbb{Z}_p\right)=n(p-1)$ (see Proposition~\ref{compu-secat-conf} or \cite[Proposition 55, p. 122]{schwarz1966}). Thus, Theorem~\ref{thm:q-bup-scat} imply the following statement. 

\begin{proposition}\cite[Theorem 30, p. 135]{schwarz1966}\label{prop:p-bup-secat}
     Suppose that $X$ is a paracompact free $\mathbb{Z}_p$-space. We have that the following statements are equivalent.
   \begin{enumerate}
       \item[(i)] $\mathrm{secat}\left(X\to X/\mathbb{Z}_p\right)> n(p-1)$.
        \item[(ii)] the triple $\left((X,\mathbb{Z}_p);\mathbb{R}^n\right)$ satisfies the $p$-th BUP.
   \end{enumerate}
\end{proposition}

On the other hand, Theorem~\ref{theorem:general-secat-index} together with Theorem~\ref{thm:lower-bound-index} imply a relationship between $\mathrm{secat}\left(X\to X/\mathbb{Z}_p\right)$ and $\text{ind}_q\left(X,\mathbb{Z}_p\right)$.

\begin{theorem}\label{prop:secat-indexp}
 Let $2\leq q\leq p$, with $p$ prime. Suppose that $X$ is a paracompact free $\mathbb{Z}_p$-space. Set $n_q=\mathrm{ind}_q\left(X,\mathbb{Z}_p\right)<\infty$. Then   
 \[n_q(q-1)+1\leq\mathrm{secat}\left(X\to X/\mathbb{Z}_p\right)\leq n_q(p-1)+q-1.\] 
 Or equivalently:
   \[\mathrm{secat}\left(X\to X/\mathbb{Z}_p\right)=n_q(q-1)+1+j, \text{ for some } j\in\{0,1,\ldots,\epsilon_{p,q}\},\] where $\epsilon_{p,q}=n_q(p-q)+q-2$. 
\end{theorem}

\section{Applications}\label{sec:appl}
We begin by provide  a positive answer to the conjecture stated in  the Introduction,  in fact under weaker hypothesis than the ones  stated  there. Then we compute secat of double covering of closed surfaces and double covering of the  $3$-manifolds having geometry  $S^2\times \mathbb{R}$. Furthermore, we compute secat of $|\mathbb{Z}_n|$-covering of closed surfaces. 
 
\medskip Set $\epsilon_{p,q}$ as given in Theorem~\ref{prop:secat-indexp}. Note that $\epsilon_{p,q}=0$ if and only if $p=q=2$. So, we obtain the following result.

\begin{corollary}\label{cor:conj}
  Suppose that $X$ is a paracompact free $\mathbb{Z}_2$-space. Then \[\mathrm{secat}\left(X\to X/\mathbb{Z}_2\right)=\text{ind}_2\left(X,\mathbb{Z}_2\right)+1.\]  
\end{corollary}
\begin{proof}
    It is a direct consequence of Theorem~\ref{prop:secat-indexp}.
\end{proof}

In addition, Theorem~\ref{prop:secat-indexp} implies the following statement which provides more accurate results than Proposition~\ref{prop:lower-secat-index}. 

\begin{corollary}\label{cor:lower-secat-lower-also-index}
 Let $2\leq q\leq p$, with $p$ prime. Suppose that $X$ is a paracompact free $\mathbb{Z}_p$-space and $\text{ind}_q\left(X,\mathbb{Z}_p\right)<\infty$.
 \begin{enumerate}
     \item Suppose that $\mathrm{secat}\left(X\to X/\mathbb{Z}_p\right)=2$. 
 \begin{enumerate}
      \item[(1)] $\text{ind}_q\left(X,\mathbb{Z}_p\right)=0$ whenever $q\geq 3$.
      %\item[(2)] If $p\geq 3$ then $n_p=0$. 
      \item[(2)] $\text{ind}_2\left(X,\mathbb{Z}_p\right)=1$. 
  \end{enumerate} 
  \item Suppose that $\mathrm{secat}\left(X\to X/\mathbb{Z}_p\right)=3$.
   \begin{enumerate}
      \item[(1)] $\text{ind}_q\left(X,\mathbb{Z}_p\right)=0$ whenever $q\geq 4$.
      \item[(2)] $\text{ind}_3\left(X,\mathbb{Z}_p\right)=1$. 
      \item[(3)] $n_2\in\{1,2\}$. 
  \end{enumerate} 
  \end{enumerate} 
\end{corollary}
\begin{proof} Set $n_q=\text{ind}_q\left(X,\mathbb{Z}_p\right)$. 
\begin{enumerate}
        \item Suppose that $\mathrm{secat}\left(X\to X/\mathbb{Z}_p\right)=2$.  By Theorem~\ref{prop:secat-indexp}, we have:
        \begin{align*}
            n_q(q-1)+1\leq 2\leq n_q(p-1)+q-1.
        \end{align*} In the case $q\geq 3$, we can conclude that $2n_q+1\leq n_q(q-1)+1\leq 2$ and thus $n_q=0$. On the other hand, $n_2+1\leq 2\leq n_2(p-1)+1$ and hence $n_2=1$.
        \item Suppose that $\mathrm{secat}\left(X\to X/\mathbb{Z}_p\right)=3$. By Theorem~\ref{prop:secat-indexp}, we have:
        \begin{align*}
            n_q(q-1)+1\leq 3\leq n_q(p-1)+q-1.
        \end{align*} In the case $q\geq 4$, we can conclude that $3n_q+1\leq n_q(q-1)+1\leq 3$ and thus $n_q=0$. On the other hand, $2n_3+1\leq 3\leq n_3(p-1)+2$ and hence $n_3=1$. Furthermore, $n_2+1\leq 3\leq n_2(p-1)+1$ and we obtain that $n_2\in\{1,2\}$.
    \end{enumerate}
\end{proof}

From Item (2) of Lemma~\ref{prop-secat-map}, the equality $\mathrm{secat}(p\times 1_Z)=\mathrm{secat}(p)$ holds for any fibration. Thus, we have the following example.

\begin{example}\label{example-final}
Given a free action of $\mathbb{Z}_p$ on $S^1$, with $p$ prime. Set $Y$ any Hausdorff space and consider the space $M=S^{1}\times Y$ equipped with the free $\mathbb{Z}_p$-action $g(x,y)=(gx,y)$, $g\in \mathbb{Z}_p$ and $(x,y)\in M$. Note that, the quotient map $\pi':M\to M/\mathbb{Z}_p$ coincides with the product $\pi\times 1_{Y}$, where $\pi:S^{1}\to S^1/\mathbb{Z}_p$ is the quotient map, and so, by Item (2) of Lemma~\ref{prop-secat-map}, we obtain that $\mathrm{secat}(\pi')=\mathrm{secat}(\pi)=2$ (the last equality follows from Example~\ref{exam:s1-zp}). 

On the other hand, by Equation~(\ref{eq:equiv}), we have the following $\mathbb{Z}_p$-equivariant map  
\begin{eqnarray*}
\xymatrix@C=2cm{ 
      M  \ar[r]^{\varphi} &  F(\mathbb{R}^2,p,2)
       }
\end{eqnarray*} where $\varphi(x,y)=\left(x,\sigma^{p-1}x,\sigma^{p-2}x,\ldots,\sigma^{2}x,\sigma x\right)$ for any $(x,y)\in M$ (here $\sigma$ is the generator of $\mathbb{Z}_p$ and therefore $\mathbb{Z}_p=\{1,\sigma,\sigma^2,\ldots,\sigma^{p-1}\}$). Then the triple $\left(M,\mathbb{Z}_p;\mathbb{R}^2\right)$ does not satisfy the $2$-th BUP (hence, it does not satisfy the $q$-th BUP, for any $q\in\{2,\ldots,p\}$).

Since $F(\mathbb{R},p,2)$ is not path-connected, we have that there is not a $\mathbb{Z}_p$-equivariant map from $M$ to $F(\mathbb{R},p,2)$, that is, see Proposition~\ref{top-bup}, the triple $\left(M,\mathbb{Z}_p;\mathbb{R}\right)$ satisfies the $2$-th BUP and thus $\mathrm{ind}_2\left(M,\mathbb{Z}_p\right)=1$. 
\end{example}

As a  straightforward  application  of Corollary~\ref{cor:conj}, we determine  the sectional category of a double covering $\pi:S_1 \to S_2$ between any two closed surfaces. This is a corollary of  \cite[Theorem 5.5]{daciberg2006} together with the equality $\mathrm{secat}(\pi)=\mathrm{ind}_{2}(X,\mathbb{Z}_2)+1$.
 
 Recall from \cite{daciberg2006} that involutions on closed surfaces are classified by means of 
 elements of $Hom(\pi_1(S),\mathbb{Z}_2)$ which are surjectives, where $S$ is a surface. For nonorientable  surface, using a 
 standard presentation of its fundamental group, a homomorphism is completely determined by the image on the generators, which are denoted by $\delta_i$ on the statement below. See \cite{daciberg2006} for more details.

 \begin{corollary}\label{cor} Let $(S, \tau )$ be a pair where $S$ is a closed surface and $\tau $ is a free $\mathbb{Z}_2$-action. The sectional category of the projection map $\pi: S\to S/\tau$ is three if one of the following conditions below holds:
\begin{enumerate}
    \item[1)]   $S$ is orientable and its Euler characteristic is congruent to 2 mod 4.
     \item[2)] $S$ is nonorientable, the Euler characteristic is congruent to 2 mod 4, and the
action $\tau$ is equivalent to one of the canonical actions (which correspond to the
subgroups given by the sequences of the form $(1, \delta_2, \delta_3, \ldots, \delta_{2r+1})$,  
where $\delta_i$ is either $\bar 1$ or  $\bar 0$).
\item[3)] S is nonorientable, the Euler characteristic is congruent to $0$ mod $4$ and $\tau$ is
equivalent to one of the canonical actions (which correspond to the subgroups
given by the sequences of the form $(1, \delta_2, \delta_3, \ldots, \delta_{2r})$, where $\delta_i$ is either $\bar 1$ or
 $\bar  0$).
\end{enumerate}
 Otherwise  it is two.
   \end{corollary} 

   \begin{remark}
      The particular case where $S$ is orientable and the quotient is nonorientable, from Corollary~\ref{cor} we obtain: When $S/\tau$ has
 Euler characteristic odd ( i.e. a connected  sum of an odd numbers of projective planes) the sectional category is 
 $3$ and when $S/\tau$ has Euler characteristic even( i.e. a connected  sum of an even  numbers of projective planes) the sectional category is two. 
   \end{remark}
   
   Following \cite{bghz} we have the following four $3$-manifolds: 
   $S^2\times S^1$,   $E=S^2  \tilde \times S^1\ $(the non-orientable  $S^2$-bundle over  $S^1$),  $\mathbb{R}P^2   \times  S^1$ and    $\mathbb{R}P^2\#\mathbb{R}P^2$. The main result  there is:
   
 \begin{theorem}\cite[Theorem 4]{bghz}\label{thm:dac}
       Up to equivalence, we have:
   \end{theorem}
   \begin{enumerate}
       \item[A)] The space $S^2 \times S^1$ admits four free involutions $\tau_1, \tau_2, \tau_3, \tau_4$ such that:
       \begin{enumerate}
           \item[(1)] The quotient $(S^2 \times  S^1)/\tau_1$  is homeomorphic to $S^2 \times  S^1$, and $\mathrm{ind}_{2} \left(S^2 \times  S^1,\tau_1\right) = 1$.
           \item[(2)] The quotient $(S^2 \times  S^1)/\tau_2$  is homeomorphic to $E$, and $\mathrm{ind}_{2}\left(S^2 \times  S^1,\tau_2\right) = 1$.
           \item[(3)] The quotient $(S^2 \times  S^1)/\tau_3$  is homeomorphic to $\mathbb{R}P^2  \times  S^1$, and $\mathrm{ind}_{2}\left(S^2 \times  S^1,\tau_3\right)=2$.
           \item[(4)] The quotient $(S^2 \times  S^1)/\tau_4$  is homeomorphic to   $\mathbb{R}P^3\#\mathbb{R}P^3$, and $\mathrm{ind}_{2}\left(S^2 \times  S^1,\tau_4\right)= 2$.
       \end{enumerate}
\item[B)] The space $E$ admits a unique free involution $\tau_5$. The quotient $E/\tau_5$ is homeomorphic to $\mathbb{R}P^2 \times  S^1$, and  $\mathrm{ind}_{2}\left(E,\tau_5\right)=3$.
\item[C)] The space  $ \mathbb{R}P^2 \times S^1$ admits a unique free involution $\tau_6$ . 
The quotient  $(\mathbb{R}P^2 \times S^1)/\tau_6$ is homeomorphic to 
$\mathbb{R}P^2 \times S^1$, and  $\mathrm{ind}_{2}\left( \mathbb{R}P^2   \times  S^1,\tau_6\right)=1$.
\item[D)] The space $\mathbb{R}P^2\#\mathbb{R}P^2$ admits a unique free involution $\tau_7$. The quotient $(\mathbb{R}P^2\#\mathbb{R}P^2)/\tau_7$ is homeomorphic to $\mathbb{R}P^3\#\mathbb{R}P^3$, and $\mathrm{ind}_{2}\left(\mathbb{R}P^3\#\mathbb{R}P^3,  \tau_7\right)=3$.
   \end{enumerate}
   
From Theorem~\ref{thm:dac} together with Corollary~\ref{cor:conj} follows:

\begin{corollary}\label{cor:s3-r-geometry} With the notation of Theorem~\ref{thm:dac} we obtain:
\begin{enumerate}
\item $\mathrm{secat}\left(S^2 \times S^1 \to (S^2 \times  S^1)/\tau_1\right)=2$;
\item $\mathrm{secat}\left(S^2 \times S^1 \to (S^2 \times  S^1)/\tau_2\right)=2$;
\item $\mathrm{secat}\left(S^2 \times S^1 \to (S^2 \times  S^1)/\tau_3\right)=3$;
\item $\mathrm{secat}\left(S^2 \times S^1 \to (S^2 \times  S^1)/\tau_4\right)=3$;
\item $\mathrm{secat}\left(E \to E/\tau_5\right)=4$;
\item $\mathrm{secat}\left(\mathbb{R}P^2 \times S^1 \to  (\mathbb{R}P^2 \times S^1)/\tau_6\right)=2$;
\item $\mathrm{secat}\left(\mathbb{R}P^2\#\mathbb{R}P^2  \to (\mathbb{R}P^2\#\mathbb{R}P^2)/\tau_7\right)=4$.
\end{enumerate}
\end{corollary}

Following~\cite[Theorem 1.1]{dajovi2006} we have the  following result.

\begin{theorem}\cite[Theorem 1.1, p. 1805]{dajovi2006}\label{Zn-surface} Let $M$ be a compact connected surface without boundary, and %suppose that $\mathbb{Z}_n$ acts freely on the left on $M$. % 
let $\tau : \mathbb{Z}_n \times M \to  M$ be a free action. 
Then the triple $(M,\mathbb{Z}_n; \mathbb{R}^2)$ has the Borsuk-Ulam property if and only if the
following conditions are satisfied:
\begin{enumerate}
    \item[(1)] $n \equiv  2$  mod $4$.
    \item[(2)] $M/\mathbb{Z}_n$ is non-orientable, and $(\theta_{\tau})_{Ab}(\delta)$ is non trivial.
\end{enumerate} 
\end{theorem}

From Theorem~\ref{Zn-surface} together with Proposition~\ref{thm:lower-bound-index} we obtain the following result.

\begin{corollary}\label{corz_n} For the triple $(M,\mathbb{Z}_n; \mathbb{R}^2)$ given by the theorem above which satisfies the Borsuk-Ulam property, we have $\mathrm{ind}_2(M,\mathbb{Z}_n)=2$.  Further for such $(M,\mathbb{Z}_n)$ the $\mathrm{secat}(M\to M/\mathbb{Z}_n)=3$. 
\end{corollary}
\begin{proof}
   From Proposition~\ref{thm:lower-bound-index}, the inequality $\mathrm{secat}(M\to M/\mathbb{Z}_n)\geq 3$ holds. On the other hand, we have:
  \begin{align*}
      \mathrm{secat}(M\to M/\mathbb{Z}_n)&\leq \mathrm{cat}\left( M/\mathbb{Z}_n\right)& \hbox{ (Item (iv) from Lemma~\ref{prop-secat-map})}\\
      &\leq 3 & \hbox{ (Item (iii) from Lemma~\ref{prop-secat-map})}.
  \end{align*} Hence, $\mathrm{secat}(M\to M/\mathbb{Z}_n)=\mathrm{cat}\left( M/\mathbb{Z}_n\right)=3$.
\end{proof}

\section*{Conflict of Interest Statement}
On behalf of all authors, the corresponding author states that there is no conflict of interest.

%%%%%%%%%%%%%%%%%%%%%%%%%%%%%%%%%%%%%%%%%%%%%%%%%%%%%%%%%%%%%%%%%%%%%%%%%%%%%%%%%%%%%%%%%%%%%%%%%
\bibliographystyle{plain}

\begin{thebibliography}{10}
\bibitem{aguilar2002} M. Aguilar,  S. Gitler,  and C. Prieto. Algebraic topology from a homotopical viewpoint. (New York: Springer, 2002).

\bibitem{bghz} A. Bauval, D. L. Goncalves,  C.  Hayat, and P. Zvengrowski. The Borsuk-Ulam  theorem for closed 3-manifolds having $S^2\times   \mathbb{R}$.
arXiv:2011.00657v1 [math.AT] 2 Nov 2020.
\bibitem{borsuk1933} K. Borsuk. Drei Sätze über die n-dimensionale Euklidische Sphäre. Fundam. Math. 20, 177–190  (1933).
\bibitem{cohen1976} F. Cohen, and E. L. Lusk. Configuration-like spaces and the Borsuk-Ulam theorem. Proceedings of the American Mathematical Society, 56(1), 313-317 (1976).
\bibitem{cornea2003lusternik} O. Cornea, G. Lupton, J. Oprea,  and D. Tanr{\'e}. Lusternik-Schnirelmann Category. Mathematical Surveys and Monographs, 103 (American Mathematical Society, Providence, RI, 2003).
\bibitem{fadell1962configuration} E. Fadell, and  L. Neuwirth. Configuration spaces. Math. Scand. {\bf 10} (4), 111--118 (1962).
\bibitem{daciberg2006} D.~L. Gon\c calves. The Borsuk-Ulam Theorem for surfaces. Quaestiones Mathematicae 29, 117--123 (2006).
\bibitem{dajovi2006} D.~L. Gon\c calves, J. Guaschi, and V. C. Laass. Free cyclic actions on surfaces and the Borsuk-Ulam theorem. Acta Math. Sin. (Engl. Ser.) 38, no. 10, 1803--1822 (2022). 
\bibitem{guillemin1974differential} V. Guillemin, and A. Pollack. Differential topology (Vol. 370. American Mathematical Society, 1974).
\bibitem{james1978} I. M. James. On category, in the sense of Lusternik-Schnirelmann. Topology 17.4: 331--348 (1978).
\bibitem{karasev2011} R. Karasev, and A. Volovikov. Configuration-like spaces and coincidences of maps on orbits. Algebraic \& Geometric Topology, 11(2), 1033-1052 (2011).
\bibitem{schwarz1958genus} A. S. Schwarz. The genus of a fiber space. Dokl. Akad. Nauk SSSR (NS). {\bf 119}, 219--222 (1958).
\bibitem{schwarz1966} A. Schwarz. The genus of a fiber space. Amer. Math. Soc. Transl. Ser. 2 55, 49--140 (1966).
\bibitem{steenrod1951} N. E. Steenrod. The topology of fibre bundles. IL, Miscow, 1953; Russian transl. of Princeton Math. Series, Vol. 14, Princeton Univ. Press, Princeton, New Jersey, 1951.
\bibitem{tom2008} T. tom Dieck. Algebraic topology. (Vol. 8. European Mathematical Society, 2008).
\bibitem{zapata2022} C. A. I. Zapata. Espaços de configurações no problema de planificação de movimento simultâneo livre de colisões.  Ph. D. Thesis (Universidade de São Paulo, São Paulo, 2022).
\bibitem{zapata2023} C. A. I. Zapata, D. L. Gonçalves. Borsuk–Ulam Property and Sectional Category. Bulletin of the Iranian Mathematical Society, 49(4), 41, (2023). https://doi.org/10.1007/s41980-023-00787-3
\bibitem{zapata2022higher} C. A. I. Zapata, and J. González. Higher topological complexity of a map. Turkish Journal of Mathematics, 47(6), 1616-1642 (2023). https://doi.org/10.55730/1300-0098.3453


%%%%%%%%%%%%%%%%%%%%%%%%%%%%%%%%%%%%%%%%%%%%%%%%


%%%%%%%%%%%%%%%%%%%%%%%%%%%%%%%%%%%%%%%%%%%
\end{thebibliography}

\end{document}